\documentclass{article}

\usepackage{amsmath}
\usepackage{amsthm}
\usepackage{amssymb}
\usepackage[colorlinks=true,urlcolor=blue,linkcolor=blue,plainpages=false,pdfpagelabels]{hyperref}

\newcommand{\soldetails}[1]{\begin{proof} \red{DETAILS } \quad {\em #1} \end{proof}}
\renewcommand{\soldetails}[1]{}

\newtheorem{theorem}{Theorem}[section]
\newtheorem{proposition}[theorem]{Proposition}
\newtheorem{lemma}[theorem]{Lemma}
\newtheorem{corollary}[theorem]{Corollary}

\newcommand{\limn}{\lim\limits_{n\to\infty}}
\newcommand{\N}{{\mathbb N}}

\newcommand{\rras}{\rightrightarrows}

\newcommand{\res}[1]{J_{#1}}
\newcommand{\refres}[1]{R_{#1}}

\newcommand{\rateconeq}{\sigma_{1}}
\newcommand{\ratectwoq}{\sigma_{2}}
\newcommand{\ratecthreeq}{\sigma_{3}}
\newcommand{\ratecfourq}{\sigma_{4}}
\newcommand{\ratecfiveq}{\sigma_{5}}
\newcommand{\ratecsixq}{\Lambda}

\newcommand{\norm}[1]{\left\|#1\right\|}
\newcommand{\abs}[1]{\left\vert#1\right\vert}
\newcommand{\inner}[2]{\left\langle#1, #2\right\rangle}

\title{Quadratic rates of asymptotic regularity for the Tikhonov-Mann iteration}

\author{Hora\c tiu Cheval${}^{a}$ and Lauren\c{t}iu Leu\c{s}tean${}^{a,b}$\\[2mm]
\footnotesize ${}^a$ Research Center for Logic, Optimization and Security (LOS), Department of Computer Science, \\
\footnotesize Faculty of Mathematics and Computer Science, University of Bucharest.\\
\footnotesize  Academiei 14,  010014 Bucharest, Romania\\[1mm]
\footnotesize ${}^b$ Simion Stoilow Institute of Mathematics of the Romanian Academy,\\
\footnotesize Calea Grivi\c tei 21, 010702 Bucharest, Romania \\[2mm]
\footnotesize E-mails:  andrei.cheval@s.unibuc.ro, laurentiu.leustean@unibuc.ro
}

\date{}
\begin{document}

\maketitle

\begin{abstract}
\noindent In this paper, we compute quadratic rates of asymptotic regularity for the Tikhonov-Mann iteration in $W$-hyperbolic spaces. This iteration is an extension to a nonlinear setting of 
the modified Mann iteration defined recently by Bo\c t, Csetnek and Meier in Hilbert spaces. Furthermore, we show that the Douglas-Rachford and forward–backward algorithms with Tikhonov regularization terms are special cases, in Hilbert spaces, of our  Tikhonov-Mann iteration.\\

\noindent {\em Keywords:}  Mann iteration; Tikhonov regularization; Rates of asymptotic regularity; Douglas-Rachford algorithm; forward–backward algorithm; Proof mining. \\

\noindent  {\it Mathematics Subject Classification 2010}:  47J25, 47H09, 03F10.

\end{abstract}

\section{Introduction}

Let $H$ be a Hilbert space and $T:H\to H$ be a nonexpansive mapping 
(that is, a mapping satisfying  $\| Tx - Ty \| \leq \| x - y \|$ for all $x, y \in H$). We denote by $Fix(T)$ the set of
fixed points of $T$.

One of the well-known nonlinear iterations is the \emph{Mann iteration} \cite{Man53,Gro72}, defined as follows:
\begin{equation}
x_{n + 1} = (1 - \lambda_n)x_n + \lambda_n T x_n, \label{eq:krasnoselskii-mann}
\end{equation}
where  $(\lambda_n)_{n \in \N}$ is a sequence in $[0,1]$ and $x_0\in H$ is the starting point of the iteration. 
A classical result on the asymptotic behaviour of this iteration, proved by Reich \cite{Rei79} in a more general 
setting than Hilbert spaces, 
is the following: if $Fix(T) \neq \emptyset$ and $\sum\limits_{n = 0}^{\infty} \lambda_n (1 -  \lambda_n) = \infty$, then $(x_n)$ 
converges weakly to a fixed point of $T$.

By using the so-called \emph{Tikhonov regularization terms}, studied in relation with the 
proximal point algorithm \cite{Att96,LehMou96},
Bo\c t, Csetnek and Meier \cite{BotCseMei19} introduced recently the following modified Mann iteration:
\begin{equation}
x_{n + 1} = (1 - \lambda_n) \beta_n x_n + \lambda_n T(\beta_n x_n), 
\label{eq:krasnoselskii-mann-modified} 
\end{equation}
where $(\lambda_n)_{n \in \N},(\beta_n)_{n \in \N}$ are sequences in $[0, 1]$ and $x_0\in H$.

The main result of \cite{BotCseMei19} is the strong convergence of the iteration $(x_n)$ under 
some hypotheses
on the sequences $(\lambda_n), (\beta_n)$. The first 
main step in the strong convergence proof is to obtain asymptotic regularity, 
a very important concept in nonlinear analysis and convex optimization, defined for the first time by Browder 
and Petryshyn \cite{BroPet66} for the 
Picard iteration and extended by Borwein, Reich and Shafrir \cite{BorReiSha92} to the Mann iteration. Thus, a sequence $(a_n)$ in a metric space $(X,d)$ is said to be 
{\em asymptotically regular} if $\limn d(a_n,a_{n+1})=0$
and {\em $T$-asymptotically regular} if $\limn d(a_n,Ta_n)=0$.

Let us consider in the sequel the following conditions on the sequences $(\lambda_n)$ and 
$(\beta_n)$: 
\begin{align*}
(C1) \quad  & \sum\limits_{n=0}^\infty (1 - \beta_n) =\infty, & 
(C2) \quad  &  \prod\limits_{n=0}^\infty \beta_n=0, \\
(C3) \quad  & \sum_{n = 0}^{\infty} \vert \beta_{n + 1} - \beta_n \vert < \infty, & 
(C4) \quad  & \sum_{n = 0}^{\infty} \vert \lambda_{n + 1} - \lambda_n \vert<\infty,\\
(C5) \quad  &  \limn \beta_n=1, & (C6) \quad  & \liminf\limits_{n\to\infty} \lambda_n > 0.
\end{align*}

The following result is implicit in \cite{BotCseMei19}, its proof being contained in the proof of 
\cite[Theorem~3]{BotCseMei19}.
    
\begin{theorem}\cite{BotCseMei19}\label{thm-as-reg}
Let $H$ be a Hilbert space, $T:H\to H$ be a nonexpansive mapping such that 
$Fix(T)\ne\emptyset$ and $(x_n)$ be given by \eqref{eq:krasnoselskii-mann-modified}.
Assume that the following hold:
\begin{enumerate}
\item (C1) (or, equivalently, (C2) if $\beta_n>0$ for all $n\in\N$);
\item (C3), (C4), (C5) and (C6).
\end{enumerate}
Then 
$$\limn \|x_n - x_{n+1}\| = \limn \|x_n-Tx_n \|=0.$$
\end{theorem}

A quantitative analysis of the proof of Theorem~\ref{thm-as-reg} was obtained recently by Dinis and 
Pinto (see \cite[Lemma~5]{DinPin21}).

In this paper we generalize this quantitative analysis in a two-fold way:
\begin{enumerate}
\item we consider $W$-hyperbolic spaces \cite{Koh05} instead of Hilbert spaces;
\item we extend the modified Mann iteration to the so-called {\em Tikhonov-Mann iteration} $(x_n)$ (see \eqref{def-TKM-W}). 
\end{enumerate}
Our main results (Theorems~\ref{main-thm-as-reg+T-as-reg-C1q} and \ref{main-thm-as-reg+T-as-reg-C2q})
provide uniform rates of asymptotic regularity for the  Tikhonov-Mann iteration $(x_n)$, that is rates of convergence towards $0$ of the 
sequences $(d(x_n,x_{n+1}))$, $(d(x_n,Tx_n))$. As an immediate consequence, we obtain quadratic 
rates of asymptotic regularity for the Tikhonov-Mann iteration in $W$-hyperbolic spaces.

We obtain these quantitative results by applying methods from {\em proof mining}, an 
applied area of mathematical logic, developed by Kohlenbach beginning with the 1990s (see \cite{Koh08} for 
a standard reference or \cite{Koh19} for a recent survey).

\section{$W$-hyperbolic spaces}

Let us define a {\em W-space} to be  a structure of the form $(X,d,W)$, where 
$(X,d)$ is a metric space and $W:X\times X\times [0,1]\to X$ is a function. The 
mapping $W$ was already considered by Takahashi \cite{Tak70} in the 1970s. We also say that 
$W$ is a {\em convexity mapping}, as $W(x,y,\lambda)$ could be regarded as an abstract convex 
combination of the points $x$ and $y$ with parameter $\lambda$. That is why we use in the sequel 
\begin{center}
the notation 
$(1-\lambda)x + \lambda y$ for $W(x,y,\lambda)$. 
\end{center}

A very important class of $W$-spaces are the $W$-hyperbolic spaces, introduced by 
Kohlenbach \cite{Koh05} under the name of ``hyperbolic spaces". A {\em $W$-hyperbolic space} is a $W$-space $(X,d,W)$ satisfying the following
for all $x,y,w,z\in X$ and all $\lambda,~\tilde{\lambda} \in [0,1]$:
$$
\begin{array}{ll}
\text{(W1)} & d(z,(1-\lambda)x + \lambda y) \le  (1-\lambda)d(z,x)+\lambda d(z,y),\\[1mm]
\text{(W2)} & d((1-\lambda)x + \lambda y,(1-\tilde{\lambda})x + \tilde{\lambda} y) =  \vert \lambda-\tilde{\lambda} \vert d(x,y), \\[1mm]
\text{(W3)} & (1-\lambda)x + \lambda y = \lambda y + (1-\lambda)x, \\[1mm]
\text{(W4)} & d((1-\lambda)x + \lambda z,(1-\lambda)y + \lambda w) \le (1-\lambda)d(x,y)+\lambda d(z,w).
\end{array}
$$

A normed space is a $W$-hyperbolic space: one defines simply 
$W(x,y,\lambda)=(1-\lambda)x+\lambda y$. Furthermore, important classes of geodesic spaces 
such as  Busemann spaces \cite{Pap05} or 
CAT(0) spaces \cite{AleKapPet19,BriHae99} are $W$-hyperbolic. 
In fact, Busemann spaces are exactly the uniquely geodesic 
$W$-hyperbolic spaces (see  \cite[Proposition~2.6]{AriLeuLop14}) and, as pointed out in 
\cite[p. 386-388]{Koh08}, CAT(0) spaces are the $W$-hyperbolic spaces  $(X,d,W)$ satisfying 
the following reformulation of
the Bruhat-Tits inequality:  for all $x,y,z\in X$,
$$
d^2\left(z,\frac12 x + \frac12 y\right)\leq \frac12d^2(z,x)+\frac12d^2(z,y)-\frac14d^2(x,y).
$$

Let $(X,d,W)$ be a $W$-hyperbolic space.

\begin{lemma}\label{W-hyperbolic-spaces-basic}
The following hold for all $x,y,z,w\in X$ and 
all $\lambda, \tilde{\lambda}\in [0,1]$:
\begin{enumerate}
\item \label{W10-W-hyp} 
$d(x,(1-\lambda)x + \lambda y)=\lambda d(x,y)$ and $d(y,(1-\lambda)x + \lambda y)=(1-\lambda)d(x,y)$.
\item \label{Whs-basic-unuzero} 
$1x + 0 y=x$ and $0x + 1 y=y$.
\item\label{Whs-basic-xx} 
$(1-\lambda)x + \lambda x=x$.
\item\label{W42-xyzw-ineq1} 
$d((1-\lambda)x + \lambda z,(1-\tilde{\lambda})y + \tilde{\lambda} w) \leq (1-\lambda)d(x,y)+\lambda d(z,w)+ 
\vert \lambda - \tilde{\lambda} \vert d(y,w)$. 
\item\label{W42-xzw-ineq1}
$d((1-\lambda)x + \lambda z,(1-\tilde{\lambda})x + \tilde{\lambda} w) \leq \lambda d(z,w)+ \vert \lambda - \tilde{\lambda} \vert d(x,w)$. 
\end{enumerate}
\end{lemma}
\begin{proof}
\eqref{W10-W-hyp} holds already in the setting of convex metric spaces, defined by Takahashi 
\cite{Tak70} as $W$-spaces satisfying $(W1)$. 

\eqref{Whs-basic-unuzero}, \eqref{Whs-basic-xx} follow immediately from \eqref{W10-W-hyp}.

\eqref{W42-xyzw-ineq1} Let $u=(1-\lambda)x + \lambda z$ and $v=(1-\tilde{\lambda})y + \tilde{\lambda} w$. We have that 
\begin{align*}
d(u,v) & \leq d(u,(1-\lambda)y + \lambda w) + d((1-\lambda)y + \lambda w,v) \\
&\leq (1-\lambda)d(x,y)+\lambda d(z,w) + d((1-\lambda)y + \lambda w,v) \,\, \text{by (W4)}\\
&= (1-\lambda)d(x,y)+\lambda d(z,w)+ \vert \lambda - \tilde{\lambda} \vert d(y,w) \quad \text{by (W2)}.
\end{align*}

\eqref{W42-xzw-ineq1} is obtained by letting $y=x$ in \eqref{W42-xyzw-ineq1}.
\end{proof}

For all $x,y\in X$, let us denote
$$[x,y]=\{(1-\lambda)x + \lambda y \mid \lambda\in[0,1]\}.$$ 
By Lemma~\ref{W-hyperbolic-spaces-basic}.\eqref{Whs-basic-xx}, $[x,x]=\{x\}$. 
It is easy to see that 
$(X,d,W)$ is a geodesic space and, for all $x,y\in X$,  $[x,y]$ is a geodesic segment joining 
$x$ and $y$.

A nonempty subset $C\subseteq X$  is said to be {\em convex}   if 
$[x,y]\subseteq C$ for all $x,y\in C$. Any convex subset $C$ of $X$ is, in an obvious way, a $W$-hyperbolic space.

\section{ The Tikhonov-Mann iteration}\label{tkm-section}

In the sequel, $(X,d,W)$ is a $W$-hyperbolic space, $C$ is a convex subset of $X$ 
and $T:C\to C$ is a nonexpansive mapping. Let us denote by $Fix(T)$ the set of fixed points of $T$. We assume that $T$ has fixed points, i.e.
$Fix(T)\ne\emptyset$.

We define the  {\em Tikhonov-Mann iteration} starting with $x_0\in C$ as follows:
\begin{equation} 
x_{n + 1} = (1 - \lambda_n)u_n + \lambda_n Tu_n, \label{def-TKM-W}
\end{equation}
where 
\begin{equation}
u_n =(1-\beta_n)u+\beta_nx_n, \label{def-yn-W}
\end{equation}
with $u \in C$ and $(\lambda_n)_{n\in\N}, (\beta_n)_{n\in\N}$ sequences in $[0, 1]$.

Remark that if  $X$ is a normed space, $C=X$  and $u=0$, then $u_n=\beta_nx_n$, hence
\eqref{def-TKM-W} becomes 
\eqref{eq:krasnoselskii-mann-modified}. Therefore, our iteration $(x_n)$ is a generalization to 
the nonlinear setting of $W$-hyperbolic spaces of the modified Mann iteration introduced
in  \cite{BotCseMei19}.

\begin{lemma}\label{xn-yn-bounded} 
Let $p$ be a fixed point of $T$. Define
\begin{equation}
M= \max\{d(x_0, p), d(u, p)\}. \label{def-M-bound-xn}
\end{equation}
The following hold for all $n \in \N$:
\begin{enumerate}
\item\label{i-dxnup-dxnp} $d(x_{n + 1}, p)  \leq (1 - \beta_n) d(u, p) + \beta_n d(x_n, p)$. %
\item\label{xn-bounded} $d(x_n, p) \leq M$ and $d(x_n, u)\leq 2M$.
\item\label{yn-bounded}  $d(u_n, p) \leq M$ and $d(u_n, Tu_n) \leq 2M$.
\end{enumerate}
\end{lemma}
\begin{proof}
\begin{enumerate}
\item Applying (W1) twice and the fact that $d(Tu_n, p)=d(Tu_n, Tp)\leq d(u_n, p)$,  we get that 
\begin{align*}
d(x_{n + 1}, p) & \leq  (1 - \lambda_n) d(u_n, p) + \lambda_n d(Tu_n, p) \leq d(u_n, p) \\
&  \leq (1 - \beta_n) d(u, p) + \beta_n d(x_n, p).
\end{align*}
\item Use \eqref{i-dxnup-dxnp} and induction on $n$ to get that $d(x_n, p)\leq M$ for all $n\in\N$. Obviously,  
$d(x_n, u) \leq d(x_n, p) + d(u, p) \leq 2M$.
\item By (W1) and \eqref{xn-bounded}, we have that 
$$d(u_n, p)  \leq (1 - \beta_n) d(u, p) + \beta_n d(x_n, p)\leq M.$$ 
Furthermore, 
$d(u_n, Tu_n) \leq d(u_n,p)+d(p,Tu_n) \leq 2d(u_n,p) \leq 2M$.
\end{enumerate}
\end{proof}

\begin{proposition}
Let $p$ be a fixed point of $T$ and $M$ be defined by \eqref{def-M-bound-xn}. For all $n \in \N$,
\begin{align}
d(u_{n + 1}, u_n) & \leq \beta_{n + 1} d(x_{n + 1}, x_n) + 
2M  \vert \beta_{n + 1} - \beta_n \vert, \label{dyn-consecutive}\\
d(x_{n + 2}, x_{n + 1}) & \leq \beta_{n + 1} d(x_{n + 1}, x_n) + 2M \left(\vert \beta_{n + 1} - \beta_n \vert+
\vert\lambda_{n + 1} - \lambda_n \vert\right),
\label{dxn-consecutive}\\
d(x_n,u_n) & =  (1-\beta_n)d(u,x_n), \label{dxnyn}\\
\begin{split}\label{dxn-Txn}
d(x_n, T x_n) & \leq  d(x_n, x_{n + 1}) + \lambda_n(1-\beta_n)d(u,x_n) \\
& + (1 - \lambda_n)\beta_nd(x_n,Tx_n) + (1 - \lambda_n)(1-\beta_n)d(u,Tx_n), 
\end{split}\\
\lambda_n d(x_n, T x_n) & \leq d(x_n, x_{n + 1}) + 2M(1-\beta_n). \label{useful-T-as-reg}
\end{align}
\end{proposition}
\begin{proof}
We have that 
\begin{align*}
d(u_{n + 1}, u_n) 
&\leq \beta_{n + 1} d(x_{n + 1}, x_n) +  \vert \beta_{n + 1} - \beta_n \vert d(u, x_n) \quad 
\text{by Lemma~\ref{W-hyperbolic-spaces-basic}.\eqref{W42-xzw-ineq1}}\\
&\leq \beta_{n + 1} d(x_{n + 1}, x_n) + 2M \vert \beta_{n + 1} - \beta_n \vert \quad
 \text{by Lemma~\ref{xn-yn-bounded}.\eqref{xn-bounded}}.      
\end{align*}
Thus, \eqref{dyn-consecutive} holds. \\[1mm]
We prove  \eqref{dxn-consecutive} as follows:
\begin{align*}
d(x_{n + 2}, x_{n + 1}) 
& \leq (1 - \lambda_{n + 1}) d(u_{n + 1}, u_n) + \lambda_{n + 1} d(T u_{n + 1}, T u_n)  \\
& \quad + \vert \lambda_{n + 1} - \lambda_n \vert d(u_n, T u_n) \qquad 
\text{by Lemma~\ref{W-hyperbolic-spaces-basic}.\eqref{W42-xyzw-ineq1}}\\
& \leq  d(u_{n + 1}, u_n) + \vert  \lambda_{n + 1} - \lambda_n \vert  d(u_n, T u_n) \\
&\leq \beta_{n + 1} d(x_{n + 1}, x_n) + 2M \left(\vert  \beta_{n + 1} - \beta_n \vert  + 
 \vert \lambda_{n + 1} - \lambda_n \vert \right) \\
& \quad \text{by \eqref{dyn-consecutive} and Lemma~\ref{xn-yn-bounded}.\eqref{yn-bounded}}.
\end{align*}
\eqref{dxnyn} follows immediately from \eqref{def-yn-W} and 
Proposition~\ref{W-hyperbolic-spaces-basic}.\eqref{W10-W-hyp}.\\[1mm]
Furthermore, 
\begin{align*}
d(x_n, T x_n) & \leq d(x_n, x_{n + 1}) + d(x_{n + 1}, T x_n) \\ 
& \leq d(x_n, x_{n + 1}) + (1 - \lambda_n) d(u_n,Tx_n) + \lambda_n d(u_n,x_n) \\
& \quad \text{by  \eqref{def-TKM-W}, (W1) and the nonexpansiveness of~}T  \\
& \leq d(x_n, x_{n + 1}) + \lambda_n d(u_n,x_n) + (1 - \lambda_n)d(u_n, (1-\beta_n)u+\beta_n Tx_n)  \\
& \quad + (1 - \lambda_n)d((1-\beta_n)u+\beta_n Tx_n,Tx_n) \\
& =  d(x_n, x_{n + 1}) + \lambda_n d(u_n,x_n) + (1 - \lambda_n)d(u_n, (1-\beta_n)u+\beta_n Tx_n)  \\
& \quad + (1 - \lambda_n)(1-\beta_n)d(u,Tx_n) \quad \text{by Proposition~\ref{W-hyperbolic-spaces-basic}.\eqref{W10-W-hyp}}\\
&= d(x_n, x_{n + 1}) + \lambda_n(1-\beta_n)d(u,x_n) + (1 - \lambda_n)(1-\beta_n)d(u,Tx_n)\\
& \quad+ (1 - \lambda_n)d((1-\beta_n)u+\beta_nx_n,(1-\beta_n)u+\beta_n Tx_n) \\
& \quad \text{by \eqref{dxnyn} and \eqref{def-yn-W}} \\
& \leq d(x_n, x_{n + 1}) + \lambda_n(1-\beta_n)d(u,x_n) + (1 - \lambda_n)(1-\beta_n)d(u,Tx_n)\\
& \quad + (1 - \lambda_n)\beta_nd(x_n,Tx_n)   \quad \text{by (W4)}.
\end{align*}
Therefore, \eqref{dxn-Txn} is satisfied.\\[1mm]
Finally, let us prove \eqref{useful-T-as-reg}. Applying \eqref{dxn-Txn} and the fact that
 $d(u,Tx_n) \leq d(u,x_n)+d(x_n,Tx_n)$ we get that 
 \begin{align*}
d(x_n, T x_n) & \leq  d(x_n, x_{n + 1}) + \lambda_n(1-\beta_n)d(u,x_n) + 
(1 - \lambda_n)\beta_nd(x_n,Tx_n) \\
& \quad + (1 - \lambda_n)(1-\beta_n)d(u,x_n) + (1 - \lambda_n)(1-\beta_n)d(x_n,Tx_n) \\
&= d(x_n, x_{n + 1}) + (1-\beta_n)d(u,x_n) + (1 - \lambda_n)d(x_n,Tx_n).
\end{align*}
Move $(1 - \lambda_n) d(x_n, T x_n)$ to the left-hand side and apply Proposition~\ref{xn-yn-bounded}.\eqref{xn-bounded} 
to obtain that 
\begin{align*}
\lambda_n d(x_n, T x_n) & \leq d(x_n, x_{n + 1})  + (1-\beta_n)d(u,x_n) 
\leq d(x_n, x_{n + 1}) + 2M(1-\beta_n).
\end{align*}
\end{proof}

\section{Main theorems}\label{main-theorems}

The main results of the paper are effective versions of a generalization of Theorem~\ref{thm-as-reg}, 
providing uniform rates of asymptotic regularity, in the setting of $W$-hyperbolic spaces, for 
the Tikhonov-Mann iteration $(x_n)$ defined by \eqref{def-TKM-W}.

Before giving the main theorems, we recall some quantitative notions. Let $(a_n)_{n\in\N}$ be a 
sequence in  a metric space $(X,d)$, $a\in X$ and $\varphi:\N\to\N$. 
If $\limn a_n =a$, then $\varphi$ is a  {\em rate of convergence} for $(a_n)$ (towards $a$) if
\[\forall k\in\N\, \forall n\geq \varphi(k) \left(d(a_n,a)\leq \frac1{k+1}\right).\]
 If $(a_n)$ is Cauchy, then $\varphi$ is a {\em Cauchy modulus} for $(a_n)$ if 
\[\forall k\in\N\, \forall n\geq \varphi(k)\, \forall p\in\N \left(d(a_{n+p},a_n)\leq 
\frac1{k+1}\right).\]
Let $(b_n)_{n\in\N}$ be a sequence of nonnegative real numbers. If the series $\sum\limits_{n=0}^\infty b_n$ 
diverges, then a  {\em rate of divergence} of the series is a function $\theta:\N\to\N$ satisfying  
$\sum\limits_{i=0}^{\theta(n)} b_i \geq n$ for all $n\in\N$. A Cauchy modulus of a convergent series $\sum\limits_{n=0}^\infty b_n$ 
is a Cauchy modulus of the sequence $\left(\sum\limits_{i=0}^{n} b_i\right)_{n\in\N}$ and a rate of convergence
of a convergent product $\prod\limits_{n=0}^\infty b_n$  is a rate of convergence of the sequence 
$\left(\prod\limits_{i=0}^n b_i\right)_{n\in\N}$. \\[1mm]

We consider in the following quantitative versions of (C1)-(C6):\\[1mm]
\begin{tabular}{lll}
$(C1_q)$ & $\sum\limits_{n=0}^\infty (1 - \beta_n)$ diverges with rate of divergence $\rateconeq$; \\[3mm]
$(C2_q)$ & $\prod\limits_{n=0}^\infty \beta_{n+1}=0$ with rate of convergence $\ratectwoq$; \\[3mm]
$(C3_q)$ & $\sum\limits_{n=0}^\infty  \vert \beta_{n + 1} - \beta_n \vert$ converges with Cauchy modulus $\ratecthreeq$; \\[3mm]
$(C4_q)$ & $\sum\limits_{n=0}^\infty  \vert \lambda_{n + 1} - \lambda_n \vert$ converges with Cauchy modulus $\ratecfourq$; \\[3mm]
$(C5_q)$ & $\limn \beta_n=1$ with rate of convergence $\ratecfiveq$; \\[3mm]
$(C6_q)$ & $\ratecsixq\in\N^*$ and $N_{\ratecsixq}\in\N$ are such that $\lambda_n \geq \frac1\ratecsixq$ 
for all $n\geq N_{\ratecsixq}$. \\[1mm]
\end{tabular}

\mbox{}

In the sequel, $(X,d,W)$ is a $W$-hyperbolic space, $C$ is a convex subset of $X$, 
$T:C\to C$ is a nonexpansive mapping with $Fix(T)\ne \emptyset$, and $(x_n)$ is defined by \eqref{def-TKM-W}. 
If $\limn d(x_n,x_{n+1})=0$, $(x_n)$ is said to be 
{\em asymptotically regular} and a rate of convergence of  $(d(x_n,x_{n+1}))$ towards $0$ is called a 
{\em rate of asymptotic regularity} for $(x_n)$. Furthermore, we say that $(x_n)$ is {\em $T$-asymptotically regular}
if $\limn d(x_n,Tx_n)=0$; a {\em rate of $T$-asymptotic regularity} for $(x_n)$ is a rate of convergence 
of  $(d(x_n,Tx_n ))$ towards $0$.

\mbox{}

Our first quantitative result is a generalization of \cite[Lemma~5]{DinPin21} to our setting.

\begin{theorem}\label{main-thm-as-reg+T-as-reg-C1q}
Assume that $(C1_q)$, $(C3_q)$, $(C4_q)$ hold and let $K\in\N^*$ be such that $K\geq M$, where $M$ is given by 
\eqref{def-M-bound-xn} for some $p\in Fix(T)$. Define 
\begin{equation}\label{def-main-chi}
\chi:\N\to\N, \quad  \chi(k)=\max\{\ratecthreeq(8K(k+1)-1), \ratecfourq(8K(k+1)-1)\}. 
\end{equation}
The following hold:
\begin{enumerate}
\item\label{main-thm-as-reg-C1q} $(x_n)$  is asymptotically regular with rate of asymptotic regularity $\Sigma$ defined by 
\begin{equation}\label{def-rate-as-reg-C1q}
 \Sigma(k)=\rateconeq(\chi(3k+2)+2+\lceil \ln(6K(k+1))\rceil)+1. 
\end{equation}
\item\label{main-thm-T-as-reg-C1q}  If, furthermore, $(C5_q)$ and $(C6_q)$ hold, then 
$(x_n)$  is $T$-asymptotically regular with rate of 
$T$-asymptotic regularity $\Phi$ defined by 
\begin{equation}\label{def-rate-T-as-reg-C1q}
 \Phi(k)= \max\{N_{\ratecsixq}, \Sigma(2\ratecsixq(k+1)-1), \ratecfiveq(4K\ratecsixq(k+1)-1)\}.
\end{equation}
\end{enumerate}
\end{theorem}

We obtain a second quantitative result by taking as a hypothesis $(C2_q)$ instead of $(C1_q)$.

\begin{theorem}\label{main-thm-as-reg+T-as-reg-C2q}
Assume that  $(C2_q)$, $(C3_q)$, $(C4_q)$ hold,  $\beta_n>0$ for all $n\in\N$ and let 
$K$, $\chi$ be as in the hypothesis of Theorem~\ref{main-thm-as-reg+T-as-reg-C1q}.
Suppose, moreover, that, if we denote 
$P_n=\prod\limits_{i=0}^n \beta_{i+1}$, there exists a mapping 
$\psi_0:\N\to\N^*$ such that $\frac1{\psi_0(k)}\leq P_{\chi(3k+2)}$ for all $k\in\N$.

The following hold:
\begin{enumerate}
\item\label{main-thm-as-reg-C2q} $(x_n)$  is asymptotically regular with rate of asymptotic 
regularity $\widetilde{\Sigma}$ defined by 
\begin{equation}\label{def-rate-as-reg-C2q}
\widetilde{\Sigma}(k)= 
\max\left\{\ratectwoq\left(6K(k+1)\psi_0(k)-1\right),\chi(3k+2)+1\right\}+1.
\end{equation}
\item\label{main-thm-T-as-reg-C2q}   If, furthermore, $(C5_q)$ and $(C6_q)$ hold, then 
$(x_n)$  is $T$-asymptotically regular with rate of 
$T$-asymptotic regularity $\tilde{\Phi}$ defined by 
\begin{equation}\label{def-rate-T-as-reg-C2q}
\tilde{\Phi}(k)=
\max\{N_{\ratecsixq}, \widetilde{\Sigma}(2\ratecsixq(k+1)-1), \ratecfiveq(4K\ratecsixq(k+1)-1)\}.
\end{equation}
\end{enumerate}
\end{theorem}

The proofs of Theorems~\ref{main-thm-as-reg+T-as-reg-C1q}
and \ref{main-thm-as-reg+T-as-reg-C2q} are given in Section~\ref{main-thms-proofs}. By forgetting 
the quantitative information, as an immediate consequence of any of 
these theorems we get the extension of Theorem~\ref{thm-as-reg} 
obtained by taking $W$-hyperbolic spaces instead of Hilbert spaces and 
by considering the iteration $(x_n)$ defined by \eqref{def-TKM-W}.

A very important feature of the rates 
of ($T$-)asymptotic regularity computed by our main theorems is 
their extremely weak dependency on the $W$-hyperbolic space $(X,d,W)$, 
the points $x_0,u\in C$ and the mapping $T$: 
\begin{center}
only through $K \geq \max\{d(x_0, p), d(u, p)\}$ 
(where $p$ is a fixed point of $T$). 
\end{center}
It follows that for  bounded sets $C$ it suffices to take $K$ to be an upper bound for the diameter $\text{diam}(C)$ of $C$.

The dependency on the sequences $(\lambda_n)$, $(\beta_n)$ is given by  $(C1_q)-(C6_q)$. However, for the example 
we present below, the rates appearing in these quantitative hypotheses can be easily computed.

As a consequence of Theorem~\ref{main-thm-as-reg+T-as-reg-C2q} we get the following. 

\begin{corollary}\label{TM-example-quadratic}
Let $\lambda_n=\lambda \in (0,1]$ and $\beta_n = 1-\frac1{n+1}$ for every $n\in\N$. Then
\begin{equation}\label{rate-as-reg-example}
    \Sigma_0(k) = 144 K^2(k+1)^2 + 6K(k + 1)
\end{equation}
is a rate of asymptotic regularity for $(x_n)$, and 
\begin{equation}\label{rate-T-as-reg-example}
 \Phi_0(k) =  
576 K^2\left\lceil\frac1\lambda \right\rceil^2(k+1)^2+12K\left\lceil\frac1\lambda \right\rceil(k+1)
\end{equation}
is a rate of $T$-asymptotic regularity for $(x_n)$.
\end{corollary}
\begin{proof} We obtain \eqref{rate-as-reg-example} as a consequence of 
Theorem~\ref{main-thm-as-reg+T-as-reg-C2q}.\eqref{main-thm-as-reg-C2q}. Remark first that for all $n\in\N$,
\[
P_n = \prod\limits_{i=0}^n \beta_{i+1} = \frac1{n+2} \quad  \text{and} \quad 
\sum\limits_{i=0}^n  \vert \beta_{i + 1} - \beta_i \rvert  =1 - \frac1{n+2}.
\]
It follows immediately that $(C2_q)$  and   $(C3_q)$ hold with $\ratectwoq(k)=\ratecthreeq(k)=k$. Obviously, $(C4_q)$  holds with $\ratecfourq(k)=0$. We get that $\chi(k)=8K(k + 1) - 1$, so 
we can take $\psi_0(k) = \chi(3k + 2) + 2 = 24K(k + 1) + 1$. Applying \eqref{def-rate-as-reg-C2q}, it follows that 
$$ \widetilde{\Sigma}(k)=144 K^2(k+1)^2 + 6K(k + 1)=\Sigma_0(k). $$

Furthermore, $(C5_q)$  holds with $\ratecfiveq(k)=k$ and $(C6_q)$ holds with 
$\ratecsixq=\left\lceil\frac1\lambda \right\rceil$ and  
$N_{\ratecsixq}=0$. 
Apply now Theorem~\ref{main-thm-as-reg+T-as-reg-C2q}.\eqref{main-thm-T-as-reg-C2q} to get that
$$\tilde{\Phi}(k)= 
576 K^2\left\lceil\frac1\lambda \right\rceil^2(k+1)^2+12K\left\lceil\frac1\lambda \right\rceil(k+1)
= \Phi_0(k).$$
\end{proof}

Thus, for $\lambda_n=\lambda \in (0,1]$ and $\beta_n = 1-\frac1{n+1}$, we get quadratic 
rates of ($T$-)asymptotic regularity for the Tikhonov-Mann iteration $(x_n)$.

We remark that if we use Theorem~\ref{main-thm-as-reg+T-as-reg-C1q} instead of Theorem~\ref{main-thm-as-reg+T-as-reg-C2q} 
for the above example (that is, we apply $(C1_q)$ instead of $(C2_q)$), we obtain exponential rates of 
($T$-)asymptotic regularity, due to the fact the series 
$\sum\limits_{n=0}^\infty (1 - \beta_n)=\sum\limits_{n=0}^\infty \frac1{n+1}$ has an exponential rate of divergence. 
The idea to replace $(C2_q)$ with $(C1_q)$ was used by Kohlenbach \cite{Koh11} to compute, for the first time, 
quadratic rates of asymptotic regularity for the Halpern iteration in normed spaces and it was applied again 
in \cite{KohLeu12a} for the Halpern iteration in $W$-hyperbolic spaces
and in  \cite{LeuPin21} for a Halpern-type proximal point algorithm in Hilbert spaces.

The research direction of  obtaining quantitative results on the asymptotic regularity of the Mann iteration $(x_n)$ (see \eqref{eq:krasnoselskii-mann}) has a long history.  
Quadratic rates  of asymptotic regularity were computed  by Baillon and Bruck \cite{BaiBru96} in normed spaces. 
For $W$-hyperbolic spaces, only exponential rates were  obtained by Kohlenbach and the second author \cite{KohLeu03}.  Applying proof mining to 
Groetsch's proof  \cite{Gro72} of the asymptotic regularity in uniformly convex Banach spaces, the second author computed in \cite{Leu07} rates of asymptotic regularity  in 
$UCW$-hyperbolic spaces \cite{Leu10}, a class of uniformly convex geodesic spaces; as an immediate corollary,  one gets quadratic rates in the setting of $CAT(0)$ spaces.

Baillon and Bruck conjectured in \cite{BaiBru96} the existence of a constant $\kappa$  such that, for bounded sets $C$ in normed spaces, 
\begin{equation} 
\norm{x_n-Tx_n}\leq \kappa\,\frac{\text{diam}(C)}{\sum_{i=1}^n \lambda_i(1-\lambda_i)}. \label{estimate-M-BB}
\end{equation}
They showed that, for constant $\lambda_n=\lambda\in(0,1)$, one can take $\kappa=\frac1{\sqrt{\pi}}$.
Cominetti, Soto and Vaisman \cite{ComSotVai14} settled Baillon and Bruck's conjecture by proving that \eqref{estimate-M-BB} holds  with $\kappa=\frac1{\sqrt{\pi}}$ for general $(\lambda_n)$; the constant $\kappa=\frac1{\sqrt{\pi}}$ was showed to be tight by Bravo and Cominetti \cite{BraCom18}.  

These estimates were extended by Bravo, Cominetti and Pavez-Sign\'{e} \cite{BraComPav19}  to inexact versions of the Mann iteration. As a consequence, one obtains quadratic rates of asymptotic regularity for these inexact versions in Banach spaces; in  the setting of Hilbert spaces, quadratic rates  were computed previously by Liang, Fadili, and Peyr\'e \cite{LiaFadPey16}.

\section{Proofs of the main theorems}\label{main-thms-proofs}

The proof of Theorem~\ref{thm-as-reg} uses, in an essential way, the following lemma. 

\begin{lemma}\label{lemma-Xu02-2}
Let $(a_n)_{n\in\N}$ be a sequence in $[0,1]$ and 
$(c_n)_{n\in\N}, (s_n)_{n\in\N}$ sequences of nonnegative real numbers satisfying, for all $n\in\N$,
\begin{equation}
s_{n+1}\leq (1-a_n)s_n + c_n. \label{def-sn-an-cn}
\end{equation}
Assume that $\sum\limits_{n=0}^\infty a_n$ diverges 
(or, equivalently, $\prod\limits_{n=0}^\infty (1-a_n)=0$) 
and $\sum\limits_{n=0}^\infty c_n$ converges. Then $\limn s_n =0$.
\end{lemma}

The above lemma is a particular case of \cite[Lemma~2.5]{Xu02}, whose quantitative versions were proved
in \cite[Section~3]{LeuPin21}. The following  quantitative version of Lemma~\ref{lemma-Xu02-2}, which is
an immediate consequence of the results from \cite{LeuPin21}, is the main tool in the proofs 
of our main theorems.

\begin{proposition}\label{quant-lem-Xu02-bn-0}
Let $(a_n)$ be a sequence in $[0,1]$ and 
$(c_n), (s_n)$ sequences of nonnegative reals  such that  \eqref{def-sn-an-cn} holds 
for all $n\in\N$. Assume that $L\in\N^*$ is an upper bound on $(s_n)$ and that $\sum\limits_{n=0}^\infty c_n$ converges with 
Cauchy modulus $\chi$.
\begin{enumerate}
\item\label{quant-lem-Xu02-bn-0-C1q} If $\sum\limits_{n=0}^\infty a_n$  diverges with rate of 
divergence $\theta$, then $\limn s_n=0$ with rate of convergence $\Sigma$ defined by 
\begin{equation}
\Sigma(k)=\theta\big(\chi(3k+2)+1+\lceil \ln(3L(k+1))\rceil\big)+1.
\end{equation}
\item\label{quant-lem-Xu02-bn-0-C2q} Assume that $a_n<1$ for all $n\in\N$, 
$\prod\limits_{n=0}^\infty (1-a_n)=0$ with 
rate of convergence $\gamma$ and denote $A_n=\prod\limits_{i=0}^n (1-a_i)$ for all $n\in\N$. 
Suppose, furthermore, that $\delta_0:\N\to\N^*$ is such that $\frac1{\delta_0(k)}\leq A_{\chi(3k+2)}$ for all $k\in\N$.

Then 
$\limn s_n=0$ with rate of convergence $\widetilde{\Sigma}$ defined by 
\begin{equation}
\widetilde{\Sigma}(k)=
\max\left\{\gamma\left(3L(k+1)\delta_0(k)-1\right),\chi(3k+2)+1\right\}+1.
\end{equation}
\end{enumerate}
\end{proposition}
\begin{proof} Apply \cite[Propositions~2,~3]{LeuPin21} with $\delta(k)=\chi(3k+2)+1$, as one can take 
$\psi(n)=0$ for all $n\in\N$ in \cite[Lemma~7]{LeuPin21}.
\end{proof}

\mbox{}

Let $M$, $K$ be as in the hypothesis of Theorem~\ref{main-thm-as-reg+T-as-reg-C1q}. Denote  
\begin{align*} 
& L=2K, \quad s_n = d(x_{n + 1}, x_n),  \quad  a_n = 1 - \beta_{n+1}, \\ 
& c_n = 2M\left( \vert \beta_{n + 1} - \beta_n \vert + \vert \lambda_{n + 1} - \lambda_n \vert \right),
 \quad \tilde{c}_n =\sum\limits_{i=0}^n c_i.
\end{align*} 
The fact that \eqref{def-sn-an-cn} holds for all $n\in\N$ follows from \eqref{dxn-consecutive}. 
Furthermore, $L$ is an upper bound on $(s_n)$, as $d(x_{n+1},x_n)\leq d(x_{n+1},p)+ d(x_n, p) \leq 2M$, by 
Lemma~\ref{xn-yn-bounded}.\eqref{xn-bounded}. 

\begin{lemma}\label{chi-Cauchy-modulus}
Assume that $(C3_q)$, $(C4_q)$ hold and let $\chi$ be defined by \eqref{def-main-chi}. 
Then $\chi$ is a Cauchy modulus for $(\tilde{c}_n)$.
\end{lemma}
\begin{proof}
Denote  $\tilde{\beta}_n=\sum\limits_{i=0}^n \vert \beta_{i + 1} - \beta_i \vert$,
$\tilde{\lambda}_n=\sum\limits_{i=0}^n \vert \lambda_{i + 1} - \lambda_i\vert$ and let $k\in\N$ be arbitrary. 
Applying $(C3_q)$ and $(C4_q)$, we get that for all $n\geq \chi(k)$ and all $p\in\N$,
$$
\tilde{\beta}_{n+p} - \tilde{\beta}_n \leq \frac1{4L(k+1)}  \quad \text{ and } \quad 
\tilde{\lambda}_{n+p} - \tilde{\lambda}_n  \leq \frac1{4L(k+1)},
$$
hence
$$
\tilde{c}_{n+p} - \tilde{c}_n  = 2M\left(\tilde{\beta}_{n+p} - \tilde{\beta}_n\right) + 
2M\left(\tilde{\lambda}_{n+p} - \tilde{\lambda}_n\right) \leq \frac1{k+1}.
$$
\end{proof}

\begin{lemma}\label{theta-rate-divergence}
Assume that $(C1_q)$ holds and define $\theta(n)=\rateconeq(n+1)$. Then $\theta$ is a rate 
of divergence for $\sum\limits_{n=0}^\infty a_n$.
\end{lemma}
\begin{proof}
By  $(C1_q)$, we have that for all $n\in\N$,
\begin{align*} 
\sum\limits_{i=0}^{\theta(n)} a_i & = \sum\limits_{i=0}^{\rateconeq(n+1)} (1-\beta_{i+1}) 
\geq \sum\limits_{k=0}^{\rateconeq(n+1)} (1-\beta_k) - (1-\beta_0) \\
& \geq n+1-(1-\beta_0)\geq n.
\end{align*}
\end{proof}

\begin{proposition}\label{rate-T-as-reg-from-as-reg}
Suppose that $(C5_q)$, $(C6_q)$ hold and that $\Theta$ is a rate of asymptotic regularity for $(x_n)$. 
Define $\Theta^*$ by 
\begin{equation}\label{rate-T-as-reg-from-as-reg-def}
\Theta^*(k)= \max\{N_{\ratecsixq}, \Theta(2\ratecsixq(k+1)-1), \ratecfiveq(4K\ratecsixq(k+1)-1)\}.
\end{equation}
Then $\Theta^*$ is a rate of $T$-asymptotic regularity for $(x_n)$.
\end{proposition}
\begin{proof}
Let $k\in\N$ and $n \geq \Theta^*(k)$. Since $n\geq N_{\ratecsixq}$, 
we can apply  $(C6_q)$ to get that 
$\frac1{\lambda_n} \leq \ratecsixq$. By \eqref{useful-T-as-reg} and the definition of $K$, it follows that 
\begin{equation}\label{dxnTxn-1}
d(x_n, T x_n) \leq \ratecsixq d(x_n, x_{n + 1})+ 2K\ratecsixq(1-\beta_n).
\end{equation}
Since $n\geq \Theta(2\ratecsixq(k+1)-1)$, we have that 
\begin{equation}\label{dxnxnp1-2}
\ratecsixq d(x_n, x_{n + 1})\leq \frac1{2(k+1)}.
\end{equation}
As $n\geq \ratecfiveq(4K\ratecsixq(k+1)-1)$, we get that $1-\beta_n \leq \frac1{4K\ratecsixq(k+1)}$, hence 
\begin{equation}\label{1mbeta-3}
2K\ratecsixq(1-\beta_n) \leq \frac1{2(k+1)}.
\end{equation}
Apply \eqref{dxnTxn-1}, \eqref{dxnxnp1-2} and \eqref{1mbeta-3} to obtain that  
$\Theta^*$ is a rate of $T$-asymptotic regularity for $(x_n)$. 
\end{proof}

\subsection{Proof of Theorem~\ref{main-thm-as-reg+T-as-reg-C1q}}

It is easy to see that \eqref{main-thm-as-reg-C1q} follows from
Proposition~\ref{quant-lem-Xu02-bn-0}.\eqref{quant-lem-Xu02-bn-0-C1q} and 
Lemmas~\ref{chi-Cauchy-modulus}, \ref{theta-rate-divergence}. Apply \eqref{main-thm-as-reg-C1q} and 
Proposition~\ref{rate-T-as-reg-from-as-reg} to obtain \eqref{main-thm-T-as-reg-C1q}.

\subsection{Proof of Theorem~\ref{main-thm-as-reg+T-as-reg-C2q}}

We have that, for all $n\in\N$,  $a_n<1$  and  $P_n=\prod\limits_{i=0}^n(1-a_i)$. 
It follows, by $(C2_q)$, that $\prod\limits_{n=0}^\infty (1-a_n)=0$ 
with rate of convergence $\ratectwoq$ and that $\frac1{\psi_0(k)}\leq P_{\chi(3k+2)}=\prod\limits_{i=0}^{\chi(3k+2)}(1-a_i)$
for all $k\in\N$. 

Apply Proposition~\ref{quant-lem-Xu02-bn-0}.\eqref{quant-lem-Xu02-bn-0-C2q} and 
Lemma~\ref{chi-Cauchy-modulus} to get \eqref{main-thm-as-reg-C2q}.  Furthermore, 
we obtain \eqref{main-thm-T-as-reg-C2q} from \eqref{main-thm-as-reg-C2q} by applying again 
Proposition~\ref{rate-T-as-reg-from-as-reg}.

\section{Tikhonov versions of the Douglas-Rachford and forward-backward algorithms}

Bo\c t, Csetnek and Meier \cite{BotCseMei19} derived from their modified Mann iteration a  forward–backward and a Douglas–Rachford algorithm, both endowed with Tikhonov regularization terms. In this section, we define more general versions of these two algorithms  and we show that both of them are instances, in the setting of Hilbert spaces, of the Tikhonov-Mann iteration \eqref{def-TKM-W}.  Furthermore, we apply our main results from Section~\ref{main-theorems} to obtain uniform effective rates of asymptotic regularity.

In the sequel, $H$ is a Hilbert space with inner product $\inner{\cdot}{\cdot}$ and associated norm $\norm{\cdot}$. Let us recall some fundamental notions and results from convex optimization and monotone operator theory that will be used in this section. We refer to \cite{BauCom17} for details. 

Let $A:H\rras H$ be an arbitrary set-valued operator, characterized by its graph, $\text{gra} A = \{(x, u) \in H \times H \mid u \in A x\}$.
The inverse $A^{-1}:H\rras H$  of $A$ is given by $\text{gra}  A^{-1} = \{(u,x) \in H \times H \mid (x,u)\in \text{gra}  A\}$. The set of zeros of $A$ is $\text{zer}A =  \{x \in H \mid 0 \in Ax\}$.
If $\gamma\in\mathbb{R}$ and $B:H\rras H$ is another operator, then $\gamma A:H\rras H$ and $A+B:H\rras H$ are defined as follows: $(\gamma A)(x)=\{\gamma u \mid u \in Ax\}$ and $(A+B)(x)=\{u+v \mid u\in Ax, v \in Bx\}$.

The resolvent $\res{A}: H\rras H$  and the reflected resolvent  $\refres{A}: H\rras H$   of $A$ are given by  
$$\res{A} = ( \text{Id}+ A)^{-1},  \quad \refres{A} = 2\res{A} - \text{Id},$$ 
where $ \text{Id}: H\to H, \,  \text{Id}(x)=x$ is the identity mapping.
 
We say that $A$ is monotone if  $\inner{x - y}{u - v} \geq 0$ for all $x, y \in H$, $u \in A x$ and $v \in A y$, and that $A$ is maximally monotone if it is monotone 
and its graph is not properly contained in the graph of any other monotone operator. It is well-known that if $A$ is maximally monotone, then $\res{A}, \refres{A}:H\to H$ are single-valued mappings. 

Let $T : H \to H$ be a mapping. We say that $T$ is \emph{$\beta$-cocoercive} (for some $\beta > 0$) if $\inner{x - y}{T x - T y} \geq \beta \norm{Tx - T y}^2$ 
for all $x, y \in H$.
Furthermore, $T$ is said to be \emph{$\alpha$-averaged} (for some $\alpha \in (0, 1]$) if 
there exists a nonexpansive mapping $U : H \to H$ such that $T = (1 - \alpha) \text{Id} + \alpha U$.
Obviously,  $1$-averaged mappings are  exactly the nonexpansive ones.

\subsection{A Tikhonov-Douglas-Rachford algorithm}

Let $A:H\rras H$,  $B:H\rras H$  be maximally monotone operators such that  $\text{zer}(A + B)\neq \emptyset$ and let $\gamma > 0$.

The \textit{Tikhonov-Douglas-Rachford algorithm} is defined by the following iterative scheme:
\begin{align}
  \begin{split}
  u_n &= (1 - \beta_n) u + \beta_n x_n \\
  y_n &= \res{\gamma B} u_n\\
   z_n &= \res{\gamma A}(2y_n - u_n)\\
    x_{n + 1} & =u_n + \lambda_n(z_n - y_n)
  \end{split}\label{douglas-rachford}
  \end{align}
  with $x_0, u\in H$,   $(\lambda_n)$ a sequence in $[0,2]$ and $(\beta_n)$ a sequence  in $[0, 1]$. 
 
 By taking $u=0$, we get the  Douglas-Rachford algorithm with Tikhonov regularization terms defined in \cite[Section~4]{BotCseMei19}.

As in the proof of \cite[Theorem~10]{BotCseMei19}, one can see that 
\begin{equation}
x_{n + 1}  =   \left(1- \frac{\lambda_n}2\right)u_n + \frac{\lambda_n}2 Tu_n,
\end{equation}
where 
$$T = \refres{\gamma A}\refres{\gamma B}:H \to H.$$ 
Since reflected resolvents of maximally monotone operators are nonexpansive (see \cite[Corollary~23.11(ii)]{BauCom17},  $T$ is a nonexpansive mapping, and, by \cite[Proposition~26.1(iii)]{BauCom17}, we have that $\res{\gamma B}Fix(T)=\text{zer}(A + B)\ne \emptyset $, hence, $Fix(T)$ is nonempty.

Thus, $(x_n)$ is an instance of the Tikhonov-Mann iteration \eqref{def-TKM-W} 
with parameters $(\beta_n)$ and $\left(\frac{\lambda_n}{2}\right)$. It follows that we can apply our main results from Section~\ref{main-theorems} to obtain uniform rates of 
asymptotic regularity for $(x_n)$.

\begin{proposition}\label{As-reg-DR}
Assume that $(C3_q)$, $(C4_q)$ hold and let $K \in \N^*$ be such that $K \geq M$, where $M = \max\{\norm{x_0 - p}, \norm{u - p}\}$. 
Let $\chi$ be defined by \eqref{def-main-chi}.
\begin{enumerate}
 \item\label{As-reg-DR-1} If $(C1_q)$ holds, then $(x_n)$ is asymptotically regular with rate of asymptotic regularity $\Sigma$ defined by  
 \eqref{def-rate-as-reg-C1q}.
\item\label{As-reg-DR-2} Suppose that $(C2_q)$ holds and let $\psi_0$ as in Theorem~\ref{main-thm-as-reg+T-as-reg-C2q}. Then 
$(x_n)$ is asymptotically regular with rate of asymptotic regularity $\widetilde{\Sigma}$ defined by \eqref{def-rate-as-reg-C2q}.
\end{enumerate}
\end{proposition}
\begin{proof}
Apply  Theorems~\ref{main-thm-as-reg+T-as-reg-C1q} and \ref{main-thm-as-reg+T-as-reg-C2q}  for $(\beta_n)$ and $\left(\frac{\lambda_n}{2}\right)$ 
and remark the fact that  $\ratecfourq$ is a Cauchy modulus for the series $\sum\limits_{n=0}^\infty  \vert \frac{\lambda_{n + 1}}2 - \frac{\lambda_n}2 \vert$ too.
\end{proof}

Furthermore, we get as in Corollary~\ref{TM-example-quadratic} that, for $\lambda_n=\lambda \in (0,2]$ and 
$\beta_n = 1-\frac1{n+1}$,
\begin{equation}
\Sigma_0(k) = 144 K^2(k+1)^2 + 6K(k + 1) \label{quadratic-DR-xn}
\end{equation}
is a quadratic rate of asymptotic regularity for $(x_n)$.

\mbox{}

We compute in the sequel rates of asymptotic regularity for the sequences $(u_n)$, $(y_n)$, $(z_n)$  from the iterative scheme 
\eqref{douglas-rachford}.

\begin{lemma}
For all $n\in\N$, $ \norm{y_{n + 1} - y_n},  \norm{z_{n + 1} - z_n} \leq \norm{u_{n + 1} - u_n} $. 
\end{lemma}
\begin{proof}
Since $\res{\gamma B}$ is nonexpansive, we get immediately that 
$\norm{y_{n + 1} - y_n}\leq \norm{u_{n + 1} - u_n}$.  Furthermore, 
  \begin{align*}
    \norm{z_{n + 1} - z_n}    &= \norm{\res{\gamma A}((2 \res{\gamma B} - \text{Id}) u_{n + 1}) - \res{\gamma A}((2 \res{\gamma B} - \text{Id})  u_n)} \\ 
    &= \norm{\res{\gamma A}(\refres{\gamma B} u_{n + 1}) - \res{\gamma A}(\refres{\gamma B}  u_n)} \\ 
       &\leq \norm{u_{n + 1} - u_n},
  \end{align*}
since $\res{\gamma A}$ and $\refres{\gamma B}$ are nonexpansive.

\end{proof}

\begin{proposition}
Assume that $\Sigma$ is a rate of asymptotic regularity for $(x_n)$, $(C5_q)$ holds and let $K$ be as in the hypothesis of Proposition~\ref{As-reg-DR}. 
Then $(u_n)$, $(y_n)$ and $(z_n)$ are asymptotically regular with rate $\Omega$ defined by 
\begin{equation}
  \Omega(k) = \max\{\Sigma(2k + 1), \ratecfiveq(8K(k + 1) - 1)\}.
 \end{equation}
\end{proposition}
\begin{proof}
By the previous lemma, it is enough to prove that $\Omega$ is a rate of asymptotic regularity for $(u_n)$. Remark that 
\begin{align*} 
   \norm{u_{n + 1} - u_n}  &\leq \beta_{n + 1} \norm{x_{n + 1} - x_n} + 2M\abs{\beta_{n + 1} - \beta_n}  \quad \text{by \eqref{dyn-consecutive}} \\
    &\leq \norm{x_{n + 1} - x_n} + 2K((1-\beta_{n + 1}) + (1-\beta_n)) . 
\end{align*}
Therefore, for all $n \geq  \Omega(k)$, we get that
 \begin{align*}
   \norm{u_{n + 1} - u_n} &\leq \frac{1}{2(k + 1)} + 2K \left( \frac{1}{8K(k + 1)} + \frac{1}{8K(k + 1)} \right) = \frac{1}{k + 1}. 
 \end{align*}
\end{proof}

As an immediate consequence we obtain for $\lambda_n=\lambda \in (0,2]$ and 
$\beta_n = 1-\frac1{n+1}$ the same quadratic rate of asymptotic regularity, given by \eqref{quadratic-DR-xn}, for all sequences $(x_n)$, $(u_n)$, $(y_n)$ and $(z_n)$.

\subsection{A Tikhonov-forward-backward algorithm}

Let $A:H\rras H$ be a maximally monotone operator, $B:H\to H$ be a $\beta$-cocoercive mapping for some $\beta > 0$ and $\gamma \in (0, 2\beta]$. 
Assume that $\text{zer}(A + B) \neq \emptyset$. 

The \textit{Tikhonov- forward-backward algorithm} is defined as follows:
\begin{equation}\label{Tfb-def}
 x_0\in H, \quad  x_{n + 1} = (1 - \lambda_n) u_n + \lambda_n \res{\gamma A}(u_n - \gamma B u_n),
\end{equation}
where $u\in H$, $(\beta_n)$ is a sequence in $[0, 1]$, $\alpha=\frac{2\beta}{4\beta - \gamma}\in\left(\frac12,1\right]$, 
$(\lambda_n)$ is a sequence in $\left[0, \frac1\alpha\right]$, and $u_n = (1 - \beta_n) u + \beta_n x_n$.

The forward–backward algorithm with Tikhonov regularization terms defined in \cite[Section 3]{BotCseMei19} is obtained by letting  $u=0$ in \eqref{Tfb-def}.

Let us define  
$$T = \res{\gamma A} (\text{Id} - \gamma B):H\to H.$$
By the proof of \cite[Theorem~7]{BotCseMei19}, the mapping $T$ is $\alpha$-averaged, hence  
$T = (1 - \alpha) \text{Id} + \alpha U$ for some nonexpansive mapping  $U : H \to H$. Furthermore,  by \cite[Proposition~26.1.(iv)]{BauCom17}, $Fix(T) = \text{zer}(A + B)$ hence $Fix(T)=Fix(U)\ne \emptyset$. 
One can easily see that 
$$ x_{n + 1} = (1 - \lambda_n) u_n + \lambda_n Tu_n = (1-\alpha\lambda_n)u_n+\alpha\lambda_n U(u_n).$$
Hence, $(x_n)$ is the Tikhonov-Mann iteration associated with the nonexpansive mapping $U$ and parameters $(\beta_n)$ and $(\alpha\lambda_n)$. 

It follows immediately that Proposition~\ref{As-reg-DR} holds identically for the Tikhonov- forward-backward iteration $(x_n)$ too. We have only to remark that   if $(C4_q)$ is true for 
$(\lambda_n)$, then it is true also for $(\alpha\lambda_n)$ with the same modulus $\ratecfourq$, as $\alpha \leq 1$. 
Furthermore, for $\lambda_n=\lambda \in \left[0, \frac1\alpha\right]$ and 
$\beta_n = 1-\frac1{n+1}$, we get that the mapping $\Sigma_0$ defined by \eqref{quadratic-DR-xn} is a quadratic rate of asymptotic regularity for $(x_n)$.

\end{document}